\documentclass[12pt]{amsart}

\usepackage{amssymb, amsthm, amsmath}

\newtheorem{theorem}{Theorem}[section]
\newtheorem{lemma}[theorem]{Lemma}
\newtheorem{prop}[theorem]{Proposition}

\newtheorem{question}[theorem]{Question}

\theoremstyle{definition}
\newtheorem{definition}[theorem]{Definition}

\newtheorem{claim}{Claim}

\theoremstyle{remark}
\newtheorem{remark}[theorem]{Remark}

\numberwithin{equation}{section}

\newcommand{\D}{\mathbb{D}}
\newcommand{\cD}{\overline{\D}}

\newcommand{\C}{\mathbb{C}}
\newcommand{\T}{\mathbb{T}}
\newcommand{\al}{\alpha}
\newcommand{\ip}[2]{\langle #1, #2 \rangle}

\title{Stable symmetric polynomials and the Schur-Agler class}
\author{Greg Knese}

\address{University of Alabama, Tuscaloosa, AL, 35487-0350}

\date{\today}

\email{geknese@bama.ua.edu}


\thanks{This research was supported by NSF grant DMS-1001791}


\begin{document}
\bibliographystyle{apalike}
\maketitle

\begin{abstract} 
We call a multivariable polynomial an \emph{Agler denominator} if it
is the denominator of a rational inner function in the Schur-Agler
class, an important subclass of the bounded analytic functions on the
polydisk.  We give a necessary and sufficient condition for a
multi-affine, symmetric, and stable polynomial to be an Agler
denominator and prove several consequences.  We also sharpen a result
due to Kummert related to three variable, multi-affine, stable
polynomials.
\end{abstract} 

\section{Introduction}
We say a multivariable polynomial $p \in \C[z_1,\dots,z_n]$ is
\emph{stable} if $p$ has no zeros on the closed polydisk $\cD^n = \cD
\times \dots \times \cD$. ``Stable'' can refer to many variations on
this idea, but we will stick with this definition throughout. Stable
polynomials in their various related incarnations appear in complex
analysis, orthogonal polynomials (see \cite{bS05}), combinatorics, and
statistical mechanics (see \cite{dR10} or see \cite{dW10} for a survey
related to these last two).  In particular, the paper \cite{dR10}
focuses on the class of ``Lee-Yang polynomials'' which satisfy a
``non-strict'' form of stability, but are nonetheless closely related
to the polynomials we study here.

This article has two goals: (1) further develop properties and
examples of the Schur-Agler class on the polydisk, and (2) unify and
explore connections between the following two classical theorems
related to one variable polynomials.  (We postpone discussion of the
Schur-Agler class until Definition \ref{def:Agler}.)

\begin{theorem}[The Christoffel-Darboux formula]
Let $p \in \C[z]$ be a stable one variable polynomial of degree $d$
and write
\[
\tilde{p}(z) = z^d \overline{p(1/\bar{z})}.
\]
Then, there exist linearly independent polynomials $A_1,\dots, A_d \in
\C[z]$ such that
\[
\frac{|p(z)|^2 - |\tilde{p}(z)|^2}{1-|z|^2} = \sum_{j=1}^{d}
|A_j(z)|^2
\]
\end{theorem}

See \cite{bS05} for more information.

\begin{theorem}[Grace-Walsh-Szeg\H{o}] 
Let $p \in \C[z]$ be a stable one variable polynomial of degree $d$.
Then, the multi-affine symmetrization (defined below) $p_S \in
\C[z_1,\dots, z_d]$ of $p$ is stable.
\end{theorem}

See \cite{dW10} for more information and references.

Let us define the multi-affine symmetrization.  Set $[d] =
\{1,2,\dots, d\}$.  By \emph{multi-affine} we mean a polynomial which
has degree at most one in each variable separately.  For such
polynomials, it is convenient to replace multi-index notation with a
set theory notation.  Namely, if $\alpha \subset [d]$, then
\[
z^{\alpha} = \prod_{j \in \alpha} z_j, \quad z^{\varnothing} = 1.
\]
Now, if $p(z) = \sum_{j=0}^{d} p_j z^j$, then the \emph{multi-affine
  symmetrization} is given by
\[
p_S(z_1,\dots, z_d) = \sum_{\al \subset [d]} \binom{d}{|\al|}^{-1}
p_{|\al|} z^{\al}.
\]
with $|\al|$ denoting cardinality of $\al\subset [d]$.  The
multi-affine symmetrization of $p$ is the unique multi-affine
symmetric polynomial $p_S \in \C[z_1,\dots, z_d]$ with $p_S(z,z,\dots,
z) = p(z)$.  Notice symmetrization is performed at a specific degree.

The Grace-Walsh-Szeg\H{o} theorem can be useful in reducing questions
about multivariable stable polynomials to questions about multi-affine
stable polynomials by symmetrizing a given multivariable stable
polynomial in each variable separately. See \cite{dW10}, which is a
survey related to the works \cite{BB09a} and \cite{BB09b}.

It is not clear how to generalize the Christoffel-Darboux formula to
multivariable polynomials.  \emph{Two} variable stable polynomials satisfy a
Christoffel-Darboux-like formula.  If $p \in \C[z_1,z_2]$ is stable
and of multidegree $(d_1,d_2)$ (meaning degree $d_1$ in $z_1$ and
$d_2$ in $z_2$), then writing
\[
\tilde{p}(z_1,z_2) = z_1^{d_1} z_2^{d_2}
\overline{p(1/\bar{z_1},1/\bar{z_2})}
\]
we have for $z = (z_1,z_2)$
\[
|p(z)|^2 - |\tilde{p}(z)|^2 = (1-|z_1|^2) SOS_1(z)
+ (1-|z_2|^2) SOS_2(z)
\]
where the terms $SOS_1(z), SOS_2(z)$ are each a sum of squared moduli
of polynomials. Explicitly, there exist polynomials $A_1,\dots, A_N
\in \C[z]$, such that $SOS_1(z) = \sum_{j=1}^{N}|A_j(z)|^2$ and
$SOS_2(z)$ can be written in a similar way.  See \cite{CW99},
\cite{GW04}, or \cite{gK08a} for a proof of this formula.

This formula does not generalize straightforwardly to three or more
variables.  We give a special name to those polynomials for which it
does.

\begin{definition} \label{def:Agler}
We say a stable polynomial $p \in \C[z_1,\dots, z_n]$ of multidegree
$(d_1,\dots, d_n)$ is an \emph{Agler denominator} if the following
Christoffel-Darboux type of formula holds:
\begin{equation} \label{CDformula}
|p(z)|^2 - |\tilde{p}(z)|^2 = \sum_{j=1}^{n} (1-|z_i|^2) SOS_j(z)
\end{equation}
where each $SOS_j$ is a sum of squared moduli of polynomials in
$\C[z_1,\dots, z_n]$ and as usual $\tilde{p}(z_1,\dots, z_n) =
z_1^{d_1}\cdots z_n^{d_n} \overline{p(1/\bar{z_1}, \dots,
  1/\bar{z_n})}$.
\end{definition}

Let us explain the terminology.  Given a stable polynomial $p \in
\C[z_1,\dots, z_n]$, 
\[
\phi(z) = \frac{\tilde{p}(z)}{p(z)}
\]
is a rational inner function on the polydisk.  Inner just means $\phi$
has modulus $1$ almost everywhere on the $n$-torus $\T^n := (\partial
\D)^n$, and this holds in our case because $|p(z)| = |\tilde{p}(z)|$
for all $z \in \T^n$.  By the maximum principle, $\phi$ is in the
\emph{Schur class}, the set of bounded analytic functions on the
polydisk with supremum norm at most one.

If $p$ is an Agler denominator, then equation \eqref{CDformula} is
equivalent to $\phi$ being a member of a subclass of the Schur class
called the \emph{Schur-Agler class}, which we abbreviate to
\emph{Agler class}.  Such analytic functions $f$ satisfy the following
more universal bound:
\begin{equation} \label{vN}
||f(T_1,\dots, T_n)|| \leq 1
\end{equation}
for all $n$-tuples $(T_1,\dots, T_n)$ of commuting strict contractions
on a separable Hilbert space.  For $n=1,2$ the Schur class and the
Agler class coincide, but they differ for larger $n$.  See \cite{gK10}
for more background, including a discussion of the relationship
between \eqref{CDformula} and \eqref{vN}.  Due to \eqref{vN}, the
Agler class is natural from an operator theory perspective, yet it
remains poorly understood.  Agler class functions admit a nice
matricial representation (called a transfer function realization; see
\cite{gK10}) which also allows one to produce examples of Agler class
functions, but it still remains a difficult problem to determine
whether a given function is indeed in the Agler class.  In light of
all of this background, we state our motivating question.

\begin{question} Are multi-affine symmetric stable polynomials always
  Agler denominators?
\end{question}

A positive answer would mean a strengthened Grace-Walsh-Szeg\H{o}
theorem holds, while any conclusive answer would at least
enrich the study of the Agler class.  This paper represents partial
progress on this question, which we now summarize.

Theorem \ref{symAglerthm} gives a necessary and sufficient condition
for a multi-affine symmetric polynomial to be an Agler denominator in
terms of a certain $2^{d-1}\times 2^{d-1}$ matrix being positive
semi-definite (where $d$ is the number of variables).  

Our condition yields the following corollary.

\begin{theorem} \label{thm1}
  Let $p \in \C[z_1,\dots, z_n]$ be a multi-affine symmetric
  polynomial with $p(0,\dots, 0) \ne 0$.  Then, there exists an $r> 0$
  such that $p_r(z) := p(rz)$ is an Agler denominator.
\end{theorem}

Every polynomial with $p(0)\ne 0$ has a radius of stability (the
supremum of $r$ such that $p_r$ is stable).  (Note this concept is
called the \emph{inner radius} in \cite{dR10}.)  The above theorem
says that if we add the hypotheses multi-affine and symmetric, such
polynomials possess an ``Agler radius'' (the supremum of $r$ such that
$p_r$ is an Agler denominator) which is necessarily less than or equal
to its radius of stability.

While this theorem appears to be a modest contribution, we know of no
other non-trivial, naturally defined families of Schur class functions
which happen to be Agler class functions.  (``Trivial'' examples can
be obtained by taking convex combinations of Schur functions which
depend on only two variables. One can also construct examples by using
the earlier alluded to matricial representation of Agler class
functions.) Furthermore, our approach gives a method for constructing
sums of squares decompositions explicitly---something also not
generally well understood.  

What can be said for low numbers of variables?

It turns out that all 3 variable multi-affine stable polynomials are
Agler denominators whether symmetric or not.  This was proved in
\cite{aK89b}.  (Two decades ago the Agler class was of interest in
electrical engineering in the construction of ``wave digital filters''
in the papers \cite{aK89} and \cite{aK89b}. See also \cite{jB10}.)  We
shall give a proof of this fact in the appendix, since while it does
not follow the main thrust of this paper, it is nonetheless closely
related and we are able to sharpen Kummert's result slightly in the
following theorem.

\begin{theorem} \label{sharpkummert} 
If $p \in \C[z_1,z_2,z_3]$ is multi-affine and stable, then there
exist sums of squares terms such that
\[
|p|^2 - |\tilde{p}|^2 = \sum_{j=1}^{3} (1-|z_j|^2)SOS_j(z)
\]
where $SOS_3$ is a sum of two squares, while $SOS_1$, $SOS_2$ are sums
of four squares.
\end{theorem}

This is related to Theorem \ref{bounds} below and the main theme of
\cite{gK10}.  Theorem \ref{bounds} suggests we might have to use a sum
of four squares in each $SOS$ term above, but we can reduce one term
to only contain two squares.

In the case of four variables, our necessary and sufficient condition
from Theorem \ref{symAglerthm} can be significantly simplified.

\begin{theorem} \label{degree4thm}
If $p \in \C[z_1,z_2,z_3,z_4]$ is stable, multi-affine, and symmetric,
then $p$ is an Agler denominator if and only if
\[
8(|p_0|^2 - |p_4|^2) - (|p_1|^2 - |p_3|^2) \geq 2|p_2 \bar{p_1} -
\bar{p}_2 p_3 - 2(p_1 \bar{p_0} - \bar{p}_3 p_4)|
\]
where $p(z) = \sum_{\al \subset [4]} \binom{4}{|\al|}^{-1} p_{|\al|}
z^{\al}$.
\end{theorem}

We do not know if this condition holds automatically under the
assumption of stability.  One difficulty is that both sides of the
inequality are zero for symmetrizations of degree four polynomials
with all zeros on the circle.  These would be the typical extremal
examples on which to test the inequality, for if it failed for one of
them, it would fail for a nearby stable polynomial.

We have so far been unable to find a symmetric, stable, multi-affine
polynomial that is not an Agler denominator.  In Section
\ref{examples}, we present a few additional examples to illustrate.

\section{Preliminaries}

Let us reproduce the formula Agler denominators must satisfy:
\begin{equation} \label{CDformula2}
|p(z)|^2 - |\tilde{p}(z)|^2 = \sum_{j=1}^{n} (1-|z_j|^2) SOS_j(z)
\end{equation}

To begin our study we use the following result.

\begin{theorem}[\cite{gK10}] \label{bounds} If $p \in \C[z_1,\dots,
  z_n]$ is an Agler class denominator of multi-degree $d = (d_1,\dots,
  d_n)$, then the $SOS_j(z)$ term in \eqref{CDformula2} is a sum of
  squares of polynomials of degree at most 
\[
\begin{cases} d_j -1 \text{ in } z_j &\\
d_k \text{ in } z_k  &\text{ for } k \ne j
\end{cases}
\]
 In particular, $SOS_j$ can be written as a sum of at most
 $d_j\prod_{k\ne j}(d_k+1)$ polynomials (by dimensionality).
\end{theorem}

\begin{remark} \label{SOSremark}
It is worth explaining the last sentence, using notation we find
convenient for the rest of the paper.  We will typically write sums of
squares terms using vector polynomials.  So,
\[
SOS(z) = \sum_{j=1}^{N} |A_j (z)|^2
\]
where the $A_j \in \C[z_1,\dots, z_n]$ will be written as 
\[
SOS(z) = |A(z)|^2
\]
where $A(z) \in \C^{N}[z_1,\dots,z_n]$ is the vector polynomial $A =
[A_1,\dots, A_N]^t$.  Now, if $V = \text{span}\{A_j: j=1,\dots, N\}$
has dimension $m$, we can always rewrite $SOS(z)$ using the square of
a $\C^m$ valued vector polynomial.  Indeed, if $B_1,\dots, B_m$ is a
basis of $V$ then there is an $N\times m$ matrix $X$ such that
\[
X B(z) = A(z)
\]
where $B = [B_1,\dots, B_m]^t$.  Then,
\[
SOS(z) = |X B(z)|^2 = B(z)^* X^* X B(z)
\]
but $X^*X$ is a $m\times m$ positive semi-definite matrix and so can
be factored as $X^*X = Y^*Y$ with $Y$ a $m\times m$ matrix.  Hence,
\[
SOS(z) = |Y B(z)|^2, 
\]
a sum of $m$ squares.
\end{remark}

Using the above conventions we can rewrite the Christoffel-Darboux
formula (Thm \ref{CDformula}) as
\begin{equation} \label{vecCD}
|p(z)|^2 - |\tilde{p}(z)|^2 = (1-|z|^2)|A(z)|^2
\end{equation}
where now $A(z) = \sum_{j} A_j z^j$ is a vector polynomial.
If $p(z)= \sum_{j} p_j z^j$, then by matching coefficients of both
sides we get
\begin{equation} \label{Schur-Cohn}
p_j\bar{p_k} - \bar{p}_{d-j} p_{d-k} = \ip{A_j}{A_k} -
\ip{A_{j-1}}{A_{k-1}}.
\end{equation}
Here $\ip{v}{w} = w^*v$ is the standard inner product of complex
euclidean space (of dimension taken from context).

It is also useful (later) to point out that $|A(z)|^2 =
|\tilde{A}(z)|^2 := |z^{d-1}|^2|A(1/\bar{z})|^2$ and therefore
\begin{equation} \label{Amatsym}
\ip{A_j}{A_k} = \ip{A_{d-1-k}}{A_{d-1-j}}.
\end{equation}

\section{Symmetric multi-affine Agler denominators}
Again refer to equation \eqref{CDformula2}.

\begin{prop} 
If $p \in \C[z_1,\dots, z_d]$ is a symmetric multi-affine Agler
denominator, then:
\begin{itemize}
\item The sums of squares term $SOS_j(z)$ does not depend on $z_j$,
  and hence is a function of $\hat{z_j}$, the $d-1$-tuple of all
  variables except $z_j$.

\item The sums of squares terms can be chosen in a canonical
  way. Namely, there is a vector polynomial $B \in
  \C^{2^{d-1}}[z_1,\dots,z_{d-1}]$, such that
\[
SOS_j(z) = |B(\hat{z_j})|^2
\] 
\item Furthermore, $|B(z_1,\dots, z_{d-1})|^2$ is symmetric in
  $z_1,\dots, z_{d-1}$, and

\item $|B(z_1,\dots,z_{d-1})|^2$ is ``$\T^{d-1}$-symmetric'', meaning
\[
|B(z_1,\dots, z_{d-1})|^2 = |z_1\cdots z_{d-1}|^2|B(1/\bar{z_1},
\dots, 1/\bar{z}_{d-1})|^2
\]

\end{itemize}
\end{prop}

We emphasize that there are two types of symmetry here: symmetry in
terms of permuting the variables and symmetry in terms of reflection
across the torus, which we refer to as $\T^d$-symmetry. Also, note
that $B(z)$ itself is not typically symmetric.

\begin{proof} The first item follows from Theorem
  \ref{bounds} since $p$ has multidegree $(1,1,\dots, 1)$.  For
  example, the theorem says $SOS_1(z)$ is a sum of squares of
  polynomials with multidegrees bounded by $(0,1,1, \dots,1)$.

The second item follows from taking a given sum of squares
decomposition and averaging over all permutations of the variables.

Indeed, if $S_d$ denotes the set of permutations of $[d]$, define for
each $\sigma \in S_d$, $z \in \C^d$
\[
\sigma(z) = (z_{\sigma^{-1}(1)}, z_{\sigma^{-1}(2)}, \dots,
z_{\sigma^{-1}(n)})
\]
(this puts $z_j$ into $z_{\sigma(j)}$'s slot).

By symmetry of $p$ and $\tilde{p}$, 
\begin{align}
|p(z)|^2 - |\tilde{p}(z)|^2 &= d!^{-1} \sum_{\sigma \in S_d} \sum_{j=1}^{d}
(1-|z_{\sigma^{-1}(j)}|^2) SOS_{j}(\sigma(z)) \nonumber \\
&= d!^{-1} \sum_{\sigma \in S_d} \sum_{j=1}^{d}
(1-|z_{j}|^2) SOS_{\sigma(j)}(\sigma(z)) \nonumber \\
&= \sum_{j=1}^{d} (1-|z_{j}|^2) d!^{-1} \sum_{\sigma \in S_d}
SOS_{\sigma(j)}(\sigma(z)) \label{bysym}
\end{align}
Then, by Remark \ref{SOSremark} we may write
\[
|B(\hat{z_1})|^2= d!^{-1} \sum_{\sigma \in S_d} SOS_{\sigma(1)}(\sigma(z))
\]
where $B \in \C^{2^{d-1}}[\hat{z_1}]$. This is legitimate because each
term $SOS_{\sigma(1)}(\sigma(z))$ does not depend on $z_1$ and because
the polynomials in the sums of squares decomposition span a space of
dimension at most $2^{d-1}$ (the space in question being the
polynomials of degree at most $(0,1,1,\dots, 1)$).

Let $\tau \in S_d$.  Observe that upon writing $\widehat{\tau(z)_1} =
(z_{\tau^{-1}(2)}, \dots, z_{\tau^{-1}(d)})$ (i.e. $\tau(z)$ with the
first entry deleted) we have
\[
\begin{aligned}
|B(\widehat{\tau(z)_1})|^2 &= d!^{-1} \sum_{\sigma \in S_d} SOS_{\sigma(1)}
(\sigma(\tau(z)) \\
&= d!^{-1} \sum_{\sigma \in S_d} SOS_{\sigma\tau^{-1}(1)} (\sigma(z))
\end{aligned}
\]
which is the sums of squares term in front of $(1-|z_j|^2)$ for
$j=\tau^{-1}(1)$ as in \eqref{bysym}.  This also proves
$|B(\hat{z_1})|^2$ is symmetric by considering all $\tau$ with
$\tau(1) = 1$.

If necessary we can modify $|B|^2$ to be $\T^{d-1}$-symmetric, by
reflecting our sums of squares formula:
\[
|p(z)|^2 - |\tilde{p}(z)|^2 =
\sum_{j=1}^{d}(1-|z_j|^2)|\tilde{B}(\hat{z_j})|^2
\]
where
\[
\tilde{B}(z_1,\dots,z_{d-1}) = z_1z_2\cdots z_{d-1}
\overline{B(1/\bar{z_1},\dots, 1/\bar{z}_{d-1})}.
\]
and then averaging:
\[
|p(z)|^2 - |\tilde{p}(z)|^2 =
\sum_{j=1}^{d}(1-|z_j|^2)\frac{1}{2}(|B(\hat{z_j})|^2 +
|\tilde{B}(\hat{z_j})|^2).
\]
We can then re-factor $\frac{1}{2}(|B|^2 + |\tilde{B}|^2)$ as
a sum of at most $2^{d-1}$ squares to get sums of squares terms that
are $\T^{d-1}$-symmetric.
\end{proof}

Therefore, $p$ is an Agler class denominator if and only if we can
write
\begin{equation} \label{SOS}
|p(z)|^2 - |\tilde{p}(z)|^2 = \sum_{j=1}^{d} (1-|z_j|^2)
|B(\hat{z_j})|^2
\end{equation}
where $|B(\hat{z_j})|^2$ is symmetric and $\T^{d-1}$-symmetric in
$\hat{z_j}$.

Let us examine what this implies in terms of coefficients.  Write
\[
B(z) = \sum_{\al \subset [d-1]} B_{\al} z^\al \qquad B_{\al} \in \C^{2^{d-1}}
\]
then
\[
|B(z)|^2 = \sum_{\al, \beta} \ip{B_{\al}}{B_{\beta}} z^\al
\bar{z}^{\beta}.
\]
Also, write
\[
p(z_1,\dots, z_d) = \sum_{\al \subset [d]} \binom{d}{|\al|}^{-1} p_{|\al|} z^\al.
\]

\begin{prop}
\begin{enumerate}
\item Symmetry of $|B(z)|^2$ means each $\ip{B_{\al}}{B_{\beta}}$ only
  depends on $|\al|, |\beta|, |\al\cap \beta|$. So, we may write
\[
B^{i}_{j,k} := \ip{B_{\al}}{B_{\beta}}
\]
where $j = |\al|, k = |\beta|, i = |\al \cap \beta|$.  Notice that $i$
has the following restriction:
\[
0\leq i \leq j,k, d-1.
\]
It is convenient to declare that for other configurations, including
negative values of $i,j,k$, $B^{i}_{j,k} := 0$.

\item $\T^{d-1}$-symmetry means
\begin{equation} \label{symrelation}
B^{i}_{j,k} = B^{d-1-j-k+i}_{d-1-k, d-1-j}
\end{equation}

\item Writing $|\al| = j, |\beta| = k, |\al\cap \beta| = i$, the term
  $z^{\al} \bar{z}^{\beta}$ appears with coefficient
\[
(d - j - k+i)B^{i}_{j,k} - i B^{i-1}_{j-1,k-1}
\]
in the right hand side of \eqref{SOS}.  

\end{enumerate}
\end{prop}

\begin{proof}
(1) This is straightforward.

(2) This follows from
\[
\begin{aligned}
|B(z)|^2 &= |\tilde{B}(z)|^2 \\
&= \sum_{\al,\beta} \ip{B_{\beta}}{B_{\al}}
z^{[d-1]-\al} \bar{z}^{[d-1]-\beta} \\ &= \sum_{\al, \beta}
\ip{B_{[d-1]-\beta}}{B_{[d-1]-\al}} z^\al \bar{z}^{\beta}.
\end{aligned}
\]

(3) Looking at the right hand side of \eqref{SOS}, we pick up a copy
of $B^{i}_{j,k}$ for every $r \in \al^c\cap\beta^c$, where we use
$\al^c$ to denote the complement of $\al \subset [d]$ and note that
$|\al^c\cap \beta^c| = d-j-k+i$.  Finally, we pick up a copy of
$-B^{i-1}_{j-1,k-1}$ for every $r \in \al\cap \beta$.
\end{proof}

Equating coefficients on both sides of \eqref{SOS} we get
\begin{equation} \label{recursion}
\binom{d}{j}^{-1} \binom{d}{k}^{-1}(p_{j}\overline{p_{k}} -
\overline{p_{d-j}} p_{d-k}) = (d - j - k+i)B^{i}_{j,k} - i
B^{i-1}_{j-1,k-1}
\end{equation}
which holds independently of $i$.

The point now is that all values of $B^{i}_{j,k}$ can be solved for
explicitly in terms of the coefficients of $p$.  This is clear since
the restrictions on $i$ (in the above proposition) force $d-j-k+i$ to
be nonzero, in which case $B^{i}_{j,k}$ is expressed in terms of
$B^{i-1}_{j-1,k-1}$ and coefficients of $p$.  One can even write down
a complicated formula.  This gives a concrete necessary and sufficient
condition for $p$ to be an Agler class denominator.

\begin{theorem} \label{symAglerthm} A stable multi-affine symmetric
  polynomial $p \in \C[z_1,\dots, z_d]$
\[
p(z) = \sum_{\al \subset [d]} \binom{d}{|\al|}^{-1} p_{|\al|} z^\al
\]
 is an Agler class denominator if and only if the numbers
 $B^{i}_{j,k}$ which can be solved from \eqref{recursion} have the
 property that the $2^{d-1}\times 2^{d-1}$ matrix (indexed by subsets
 of $[d-1]$)
\[
\mathcal{B} := \left( B^{|\al\cap \beta|}_{|\al|, |\beta|}
\right)_{\al, \beta \subset [d-1]}
\]
is positive semi-definite.
\end{theorem}

\begin{proof}
The ``only if'' direction follows from the preceding discussion.  The
``if'' direction essentially follows from reversing all of the
arguments and observing that if the given matrix is positive
semi-definite then
\[
\sum_{\al,\beta \subset [d-1]} B^{|\al\cap\beta|}_{|\al|, |\beta|}
z^{\al} \bar{z}^{\beta}
\]
can be factored as $|B(z)|^2$.  
\end{proof}

Theorem \ref{thm1} follows from this.

\begin{proof}[Proof of Theorem \ref{thm1}]
We are assuming $p$ is a symmetric, multi-affine polynomial, and we
may assume $p(0)=1$.  For each $r$, set $p_r(z) := p(rz)$ construct
the matrix $\mathcal{B}(r)$ as above.  This matrix depends
continuously on $r$ and is positive definite when $r=0$. Therefore,
the matrix stays positive definite for $r$ in some interval containing
$0$.  By the previous theorem, for such $r$, $p_r$ is an Agler class
denominator. 
\end{proof}

\begin{remark}
Let us explicitly give the matrix $\mathcal{B}(0)$ from the proof
because even in this trivial case it is useful to see the sums of
squares decomposition.

Our ``polynomial'' is $p(z) = 1$ which we view as a multi-affine
polynomial of $d$ variables. So, $\tilde{p}(z) = z_1\cdots z_d$.
Solving the recurrence we get
\[
\begin{aligned}
B^{i}_{j,k} &= 0 \text{ if } j,k,i \text{ are not all equal}\\
B^{j}_{j,j} &= \frac{1}{d\binom{d-1}{j}}.\\
\end{aligned}
\]
Then, $\mathcal{B}(0)$ is diagonal and clearly
positive definite, and we get
\[
|B(z)|^2 = \sum_{\al \subset [d-1]} \frac{|z^\al|^2}{d\binom{d-1}{|\al|}}
\]
and hence
\[
1- |z_1\dots z_d|^2 = \sum_{j=1}^{d} (1-|z_j|^2) \sum_{\al \subset
  [d]\setminus \{j\}} \frac{|z^\al|^2}{d\binom{d-1}{|\al|}}
\]
\end{remark}

It turns out to be useful to apply the Christoffel-Darboux formula to 
\[
p(z,z,\dots,z) = \sum_{j=0}^{d} p_j z^j
\]
(recall that we have weighted our multi-affine polynomial's
coefficients to make this formula hold) and combine this with Theorem
\ref{symAglerthm}.  Combining formula \eqref{Schur-Cohn} with
\eqref{recursion} we get
\begin{equation} \label{recursion2}
\binom{d}{j}^{-1} \binom{d}{k}^{-1}(\ip{A_j}{A_k} -
\ip{A_{j-1}}{A_{k-1}}) = (d - j - k+i)B^{i}_{j,k} - i
B^{i-1}_{j-1,k-1}.
\end{equation}

The nice thing about this is that $\mathcal{B}$ is now expressed in
terms of the matrix $\ip{A_j}{A_k}$, which we know to be positive
semi-definite (in fact, \emph{positive} when $p$ is stable).

\section{Degree 4 case}
We investigate the degree 4 situation and prove Theorem
\ref{degree4thm}.  Let
\[
p(z_1,z_2,z_3,z_4) = \sum_{\al \subset\{1,2,3,4\}} \binom{4}{|\al|}^{-1}
p_{|\al|} z^\al
\]
which we assume to be stable.  Solving for $\mathcal{B}$ from Theorem
\ref{symAglerthm} in terms of the matrix $A_{j,k} = \ip{A_j}{A_k}$ as
in \eqref{recursion2} we get
\[
\begin{aligned}
B^{0}_{0,0} &= \frac{1}{4} A_{0,0}, & B^{1}_{1,1} & = \frac{1}{4^2}
A_{0,0} + \frac{1}{3\cdot 4^2} A_{1,1} \\ 
B^{0}_{1,0} &= \frac{1}{12}
A_{1,0}, &  B^{0}_{2,0} &= \frac{1}{12} A_{2,0} \\ 
B^{0}_{3,0} &=
\frac{1}{4} A_{3,0}, &  B^{0}_{1,1} &= \frac{1}{2\cdot 4^2} (A_{1,1} -
A_{0,0}) \\ 
B^{0}_{2,1} &= \frac{1}{6\cdot 4}(A_{2,1} - A_{1,0}), &
B^{1}_{2,1} &= \frac{1}{2\cdot 6 \cdot 4}( A_{2,1} + A_{1,0}) \\
 B^{1}_{3,1} &= \frac{1}{12} A_{2,0} & & \\
\end{aligned}
\]
The remaining values follow from the relation
\[
B^{i}_{j,k} = B^{3-j-k+i}_{3-k,3-j}.
\]
(It is also useful to recall equation \eqref{Amatsym}.)

Recall the $2^{4-1}\times 2^{4-1}$ matrix $\mathcal{B}$ is indexed by
subsets of $[3] = \{1,2,3\}$.  We will index according to the
ordering:
\[
\{ \varnothing, \{1\}, \{2\}, \{3\}, \{1,2\}, \{2,3\}, \{1,3\},
\{1,2,3\} \}
\]
It is convenient to break up $\mathcal{B}$ into blocks according to
the size of subset and factor out a $\frac{1}{4}$:

\[
\mathcal{B} = \frac{1}{4} \begin{bmatrix} S_{0,0} & S_{0,1} & S_{0,2} & S_{0,3}
  \\ S_{1,0} & S_{1,1} & S_{1,2} & S_{1,3} \\ S_{2,0} & S_{2,1} &
  S_{2,2} & S_{2,3} \\ S_{3,0} & S_{3,1} & S_{3,2} &
  S_{3,3} \end{bmatrix}
\]
So, for example $S_{2,1}$ is a $3\times 3$ matrix with rows indexed by
$\{\{1,2\}, \{2,3\}, \{1,3\}\}$ and columns indexed by
$\{\{1\},\{2\},\{3\}\}$.

Each block is now explicitly described.
\[
S_{0,0} = S_{3,3} = A_{0,0}
\]
\[
S_{0,1} = S_{1,0}^* = S_{2,3}^t = \overline{S_{3,2}} = \frac{1}{3}
A_{0,1} \begin{bmatrix} 1 & 1 & 1 \end{bmatrix}
\]
\[
S_{0,2} = S_{2,0}^* = S_{1,3}^t = \overline{S_{3,1}} = \frac{1}{3}
A_{0,2} \begin{bmatrix} 1 & 1 & 1 \end{bmatrix}
\]
\[
S_{0,3} =S_{3,0}^* = A_{0,3}
\]
\[
S_{1,1} = \begin{bmatrix} \frac{1}{4}A_{0,0} + \frac{1}{12}A_{1,1} &
  \frac{1}{8} (A_{1,1} - A_{0,0}) & \frac{1}{8} (A_{1,1} - A_{0,0})
  \\ \frac{1}{8} (A_{1,1} - A_{0,0}) & \frac{1}{4}A_{0,0} +
  \frac{1}{12}A_{1,1} & \frac{1}{8} (A_{1,1} - A_{0,0}) \\ \frac{1}{8}
  (A_{1,1} - A_{0,0}) & \frac{1}{8} (A_{1,1} - A_{0,0}) &
  \frac{1}{4}A_{0,0} + \frac{1}{12}A_{1,1}
\end{bmatrix}
\]
\[
S_{1,2} = S_{2,1}^* = \begin{bmatrix} \frac{1}{12}(A_{1,2} + A_{0,1}) &
  \frac{1}{6}(A_{1,2} - A_{0,1}) & \frac{1}{12}(A_{1,2} + A_{0,1})
  \\ \frac{1}{12}(A_{1,2} + A_{0,1}) & \frac{1}{12}(A_{1,2} + A_{0,1})
  & \frac{1}{6}(A_{1,2} - A_{0,1}) \\ \frac{1}{6}(A_{1,2} - A_{0,1}) &
  \frac{1}{12}(A_{1,2} + A_{0,1}) & \frac{1}{12}(A_{1,2} + A_{0,1})
\end{bmatrix}
\]
\[
S_{2,2} = S_{1,1}
\]
(one must be careful in the last equality because the entries are
indexed differently---$S_{1,1}$ is indexed by $\{ \{1\},\{2\},
\{3\}\}$ and $S_{2,2}$ is indexed by $\{\{1,2\}, \{2,3\}, \{1,3\}\}$).

This matrix, while complicated, has lots of symmetry, which we exploit
by conjugating by the following circulant type matrix
\[
R = 2 \begin{bmatrix} 1 & 0 & 0 & 0 \\
0 & C & 0 & 0 \\
0 & 0 & C & 0 \\
0 & 0 & 0 & 1
\end{bmatrix}
\]
where
\[
C = \begin{bmatrix} 1 & 1 & 1 \\
1 & \mu & \mu^2 \\
1 & \mu^2 & \mu \\
\end{bmatrix}
\]
and $\mu = e^{i2\pi/3}$.

To compute $R \mathcal{B} R^*$ we observe that
\[
C S_{1,0} = \begin{bmatrix} A_{1,0} \\ 0 \\ 0 \end{bmatrix}
\]
\[
C S_{1,1} C^* = \begin{bmatrix} A_{1,1} & 0 & 0 \\
0 & \frac{1}{8}(9A_{0,0} - A_{1,1}) & 0 \\
0 & 0 & \frac{1}{8}(9A_{0,0} - A_{1,1}) \end{bmatrix}
\]
\[
C S_{1,2} C^* = \begin{bmatrix} A_{1,2} & 0 & 0 \\
0 & \frac{1}{4}\mu^2(A_{1,2} - 3A_{0,1}) & 0 \\
0 & 0 &  \frac{1}{4}\mu (A_{1,2} - 3A_{0,1})
\end{bmatrix}
\]
The matrix $\mathcal{B}$ is positive semi-definite if and only if
$R\mathcal{B} R^*$ is, and after permuting index sets around
$R\mathcal{B}R^*$ is positive semi-definite if and only if the
following block matrix is
\[
\begin{bmatrix} A & 0 & 0 \\ 0 & X & 0 \\ 0
  & 0 & X^t \end{bmatrix}
\]
where 
\[
X  = \frac{1}{4} \begin{bmatrix} \frac{1}{2}(9A_{0,0} - A_{1,1}) &
  \mu(A_{2,1} - 3A_{1,0}) \\
  \mu^2(A_{1,2} - 3A_{0,1}) & \frac{1}{2}(9A_{0,0} -
  A_{1,1}) \end{bmatrix}.
\]
Since $A$ is positive, we only need $X$ positive semi-definite and
this amounts to the following inequality
\[
9A_{0,0} - A_{1,1} \geq 2|A_{2,1} - 3A_{1,0}|.
\]

If we translate this into coefficients of $p$ via \eqref{Schur-Cohn}
we get the inequality
\[
8(|p_0|^2 - |p_4|^2) - (|p_1|^2 - |p_3|^2) \geq 2|p_2 \bar{p_1} -
\bar{p}_2 p_3 - 2(p_1 \bar{p_0} - \bar{p}_3 p_4)|
\]
This proves Theorem \ref{degree4thm}.  

\section{Examples} \label{examples}
We have been unable to locate a stable multi-affine symmetric
polynomial which is not an Agler denominator.  Let us present some of
the simplest possible examples.  Consider $q(z) = 1-z$ which we can
symmetrize at any degree we like:
\[
\begin{aligned}
p_3(z_1,z_2,z_3) &= 1 - \frac{1}{3}\sum_{j=1}^{3}z_j \\
p_4(z_1,\dots,z_4) &= 1 - \frac{1}{4} \sum_{j=1}^{4} z_j \\
\dots \text{etc.}
\end{aligned}
\]
Note $q$ is not ``strictly'' stable, but this is unimportant for what we are
talking about---we really care about the existence of sums of squares
decompositions as in the definition of Agler denominators and are not
so worried about zeros on the boundary of the polydisk.

Theorem \ref{kummert} implies $p_3$ is an Agler denominator, Theorem
\ref{degree4thm} implies $p_4$ is an Agler denominator, and Theorem
\ref{symAglerthm} implies $p_5, \dots, p_{11}$ are Agler denominators
after lengthy computations (which we necessarily performed with a
computer since the computation for $p_{11}$ involves checking whether
a $2^{10}\times 2^{10}$ matrix is positive semi-definite).

So, for $d=3,\dots, 11$, all of the following rational inner functions
\[
\frac{d\prod_{j=1}^{d} z_j - \sum_{k=1}^{d} \prod_{j \ne k} z_j}{d -
  \sum_{j=1}^{d} z_j}
\]
satisfy the von Neumann inequality \eqref{vN}.

\section{Appendix: three variable multi-affine stable polynomials}

Here we give a proof of the following result due to Kummert and our
sharpening (Theorem \ref{sharpkummert}).

\begin{theorem}[\cite{aK89b}] \label{kummert} 
If $p\in \C[z_1,z_2,z_3]$ is multi-affine and stable, then $p$ is an
Agler denominator.
\end{theorem}

The proof we give is essentially Kummert's, although we have made it
less computational and have removed the use of a classical theorem of
Hilbert (viz. positive two variable degree 2 real polynomials are sums
of three squares) to prove our sharpening.

\begin{lemma} Let $t(z_1,z_2)$ be a positive trig polynomial of
  degree one in each variable.  Then, $t$ is the sum of squared moduli
  of two polynomials.
\end{lemma}

\begin{proof}
Write $t(z_1,z_2) = t_0(z_1) + t_1(z_1)z_2 + \overline{t_1(z_1)z_2}$.
Positivity implies $t_0(z_1) > 2|t_1(z_1)|$ for all $z_1 \in \T$ after
minimizing over $z_2$.  Then, the matrix
\[
T(z_1) = \begin{bmatrix} \frac{1}{2} t_0(z_1) & t_1(z_1) \\ \overline{t_1(z_1)}
  & \frac{1}{2} t_0(z_1) \end{bmatrix}
\]
is a positive matrix trig polynomial of degree one in $z_1$.  By the
matrix Fej\'er-Riesz theorem, it can be factored as $A(z_1)^* A(z_1)$
where $A(z_1)$ is a degree one $2\times 2$ matrix polynomial.  Then,
\[
t(z_1,z_2) = \begin{bmatrix} 1 & \bar{z}_2 \end{bmatrix}
T(z_1) \begin{bmatrix} 1 \\ z_2 \end{bmatrix} = \left|
A(z_1)\begin{bmatrix} 1 \\ z_2 \end{bmatrix} \right|^2
\]
which is a sum of two squares.
\end{proof}

\begin{proof}[Proof of Theorems \ref{kummert} and \ref{sharpkummert}]
Write $p(z) = a(z_1,z_2) + b(z_1,z_2) z_3$.  For $z_1,z_2 \in \T$, by
direct computation
\begin{equation} \label{lurkisom1}
|p|^2 - |\tilde{p}|^2 = (1-|z_3|^2)(|a(z_1,z_2)|^2- |b(z_1,z_2)|^2).
\end{equation}
Then, $|a(z_1,z_2)|^2- |b(z_1,z_2)|^2$ is a non-negative two variable
trig polynomial of degree one in each variable.  As $p$ is stable,
$|a|^2 - |b|^2$ is in fact strictly positive on $\T^2$, since a zero
would imply $|p(z_1,z_2, \cdot)| = |\tilde{p}(z_1,z_2,\cdot)|$ and
this would mean $z_3 \mapsto p(z_1,z_2,z_3)$ has a zero on $\T$.

By the lemma, we may write 
\[
|a(z_1,z_2)|^2- |b(z_1,z_2)|^2 = |E(z_1,z_2)|^2 \text{ on } \T^2
\]
where $E$ is a vector polynomial with values in $\C^2$.

We also remark that since $p$ is stable, $a$ is stable.  By the
maximum principle we can then conclude that 
\[
\frac{\tilde{b}(z_1,z_2)}{a(z_1,z_2)}
\]
is analytic and has modulus strictly less than one (since $|b| =
|\tilde{b}|$ on $\T^2$ and since $|a| > |b|$ on $\T^2$). In
particular, $a+\tilde{b}$ is stable.

We may polarize formula \eqref{lurkisom1} and get for $z_1,z_2 \in \T$
\begin{equation} \label{lurkisom2}
p(z_1,z_2,z_3) \overline{p(z_1,z_2,\zeta_3)} - \tilde{p}(z_1,z_2,z_3)
\overline{\tilde{p} (z_1,z_2,\zeta_3)} = (1-z_3 \bar{\zeta}_3)
|E(z_1,z_2)|^2,
\end{equation}
which we rearrange into
\[
\begin{aligned}
& p(z_1,z_2,z_3) \overline{p(z_1,z_2,\zeta_3)} + z_3 \bar{\zeta}_3
|E(z_1,z_2)|^2 \\
&= \tilde{p}(z_1,z_2,z_3)
\overline{\tilde{p} (z_1,z_2,\zeta_3)} + |E(z_1,z_2)|^2.
\end{aligned}
\]

Then, for fixed $z_1,z_2 \in \T$ and for varying $z_3$, the map
\begin{equation} \label{lurkisom3}
\begin{bmatrix} p(z_1,z_2,z_3) \\ z_3 E(z_1,z_2) \end{bmatrix}
\mapsto \begin{bmatrix} \tilde{p}(z_1,z_2,z_3)
  \\ E(z_1,z_2) \end{bmatrix}
\end{equation}
gives a well-defined isometry $V(z_1,z_2)$ (which depends on
$z_1,z_2$) from the span of the elements on the left to the span of
the elements on the right (the span taken over the above vectors as
$z_3$ varies).  
More concretely, by examining coefficients of $z_3$, we map
\begin{equation} \label{lurkisom4}
\begin{bmatrix} a(z_1,z_2) \\ 0 \\ 0 \end{bmatrix}
\mapsto \begin{bmatrix} \tilde{b}(z_1,z_2) \\ E(z_1,z_2) \end{bmatrix},
\qquad \begin{bmatrix} b(z_1,z_2) \\ E(z_1,z_2) \end{bmatrix}
\mapsto \begin{bmatrix} \tilde{a}(z_1,z_2) \\ 0 \\ 0 \end{bmatrix}.
\end{equation}

This is how the ``lurking isometry argument'' traditionally works,
however $V(z_1,z_2)$ does not extend uniquely to define a unitary on
$\C^3$ and we would like to extend $V(z_1,z_2)$ so that $V$ is
rational in $z_1,z_2$.

Write $E = [E_1, E_2]^t$.  Define $F = [-\tilde{E}_2,
  \tilde{E}_1]^t$.  Then, $\ip{F(z_1,z_2)}{E(z_1,z_2)} = 0$ which
means the vector
\[
X(z_1,z_2) = \begin{bmatrix} 0 \\ F(z_1,z_2) \end{bmatrix}
\]
is orthogonal to both the left and right sides of \eqref{lurkisom3}.
So, to extend $V$ to a rational unitary, it is only a matter of
assigning 
\begin{equation} \label{assignX}
V(z_1,z_2) X(z_1,z_2) = \phi(z_1,z_2) X(z_1,z_2)
\end{equation}
where $\phi$ is a unimodular function, in such a way that $V$ is
rational.

Kummert cleverly gives the matrix $V$ explicitly.
\begin{claim} 
Define
\[
V = \frac{1}{a} \begin{bmatrix} \tilde{b} & \tilde{E}^t \\ E
  & \frac{E\tilde{E}^t - a(\tilde{a}+b)I}{a+\tilde{b}}
\end{bmatrix}.
\]
Then, $V$ is holomorphic in $\D^2$ and unitary valued on $\T^2$, and
$V$ satisfies \eqref{lurkisom3} for $(z_1,z_2) \in \T^2$ and hence for
all $(z_1,z_2) \in \cD^2$ by analyticity.
\end{claim}

First, $V$ is holomorphic since $a$ and $a+ \tilde{b}$ are stable.
Using this definition of $V$, the fact that $V$ is unitary valued on
$\T^2$ will follow from checking that \eqref{lurkisom3} and
\eqref{assignX} hold (i.e. $V(z_1,z_2)$ performs the mapping as
indicated in \eqref{lurkisom3} and \eqref{assignX}).   

Indeed, it can be directly checked that the equivalent condition in
\eqref{lurkisom4} holds because of the relation
\[
\begin{aligned}
\tilde{E}(z_1,z_2)^tE(z_1,z_2) &= z_1z_2|E(z_1,z_2)|^2 \\
&=
z_1z_2(|a(z_1,z_2)|^2 - |b(z_1,z_2)|^2) = a \tilde{a} - b\tilde{b}.
\end{aligned}
\]
In addition, \eqref{assignX} holds because
\[
V(z_1,z_2) X(z_1,z_2) = -\frac{\tilde{a}+b}{a +\tilde{b}} X(z_1,z_2)
\]
since $\tilde{E}^tF = 0$, which is indeed a unimodular multiple of
$X$.  This proves the claim.

This means $V$ is a two variable rational matrix valued inner
function.  It was proved in \cite{aK89} (see also \cite{BSV05}) that
such functions have transfer function representations.  Namely, there
exists a $(2 + n_1 + n_2) \times (2+n_1+n_2)$ block unitary
\[
U = \begin{bmatrix} A & B \\ C & D \end{bmatrix} = \begin{bmatrix} A &
  B_1 & B_2 \\
C_1 & D_{11} & D_{12} \\
C_2 & D_{21} & D_{22} \end{bmatrix}
\]
where $B$ is a $2\times (n_1+n_2)$ matrix, $C$ is a $(n_1+n_2) \times
2$, $D$ is a $(n_1+n_2) \times (n_1+n_2)$ (all subdivided as
indicated) such that $V(z_1,z_2) = A + B d(z_1,z_2) (I - D
d(z_1,z_2))^{-1} C$ where
\[
d(z_1,z_2) = \begin{bmatrix} z_1I_1 & 0 \\ 0 & z_2 I_2 \end{bmatrix}.
\]
Here $I_1, I_2$ are the $n_1$, $n_2$-dimensional identity matrices,
respectively. 

Such a representation is equivalent to the formula
\begin{equation} \label{lurkisom5}
U \begin{bmatrix} I \\ z_1 G_1(z_1,z_2) \\ z_2 G_2(z_1,z_2) \end{bmatrix}
= \begin{bmatrix} V(z_1,z_2) \\ G_1(z_1,z_2) \\ G_2(z_1,z_2) \end{bmatrix}
\end{equation}
where $G_1, G_2$ are some $C^{n_1}$, $\C^{n_2}$ valued functions
(which can in fact be explicitly solved for).

Define 
\[
Y = \begin{bmatrix} p \\ z_3 E\end{bmatrix} \text{ and } H_j = G_j Y
  \text{ for } j =1,2.
\]
Then, 
\[
U \begin{bmatrix} I \\ z_1 G_1 \\ z_2
  G_2 \end{bmatrix} Y =
U \begin{bmatrix} Y \\ z_1 G_1 Y \\ z_2 G_2 Y \end{bmatrix} 
= U \begin{bmatrix} p \\ z_3 E \\ z_1 H_1 \\ z_2 H_2 \end{bmatrix} 
= \begin{bmatrix} V Y \\ H_1 \\ H_2 \end{bmatrix} 
= \begin{bmatrix}
  \tilde{p} \\ E \\ H_1 \\ H_2 \end{bmatrix}
\]
where the equations follow in order by: algebra, definitions of $Y,
H_j$, \eqref{lurkisom5}, and \eqref{lurkisom3}.

Since $U$ is a unitary and since
\[
U \begin{bmatrix} p \\ z_3 E \\ z_1 H_1 \\ z_2 H_2 \end{bmatrix}
= \begin{bmatrix} \tilde{p} \\ E \\ H_1 \\ H_2 \end{bmatrix}
\]
we have
\[
\begin{aligned}
& |p|^2 + |z_3|^2|E|^2 + |z_1|^2|H_1|^2 + |z_2|^2|H_2|^2 \\
& = |\tilde{p}|^2 + |E|^2 + |H_1|^2 + |H_2|^2
\end{aligned}
\]
which can be rearranged to give
\[
|p|^2 - |\tilde{p}|^2 = \sum_{j=1,2} (1-|z_j|^2) |H_j|^2 + (1-|z_3|^2)
|E|^2
\]
Even though we have not verified that $H_1$ and $H_2$ are polynomials,
this is enough to prove $p$ is an Agler denominator by \cite{gK10}.
In fact, Theorem \ref{bounds} forces $H_1$, $H_2$ to be polynomials of
multi-degree $(0,1,1)$, $(1,0,1)$ and the sums of squares $|H_1|^2,
|H_2|^2$ can be rewritten as sums of four squares each (by
dimensionality; see remark \ref{SOSremark}).
\end{proof}

\section*{Acknowledgments}
Thanks to John M$^{\text{c}}$Carthy and Jeff Geronimo for useful
discussions, and to Joseph Ball for bringing the work of A. Kummert to our
attention. 

\bibliography{symsa}

\begin{thebibliography}{}

\bibitem[Ball, 2010]{jB10}
Ball, J. (2010).
\newblock Multidimensional circuit synthesis and multivariable dilation theory.
\newblock {\em Multidimensional Systems and Signal Processing}, pages 1--18.
\newblock 10.1007/s11045-010-0123-2.

\bibitem[Ball et~al., 2005]{BSV05}
Ball, J.~A., Sadosky, C., and Vinnikov, V. (2005).
\newblock Scattering systems with several evolutions and multidimensional
  input/state/output systems.
\newblock {\em Integral Equations Operator Theory}, 52(3):323--393.

\bibitem[Borcea and Br{\"a}nd{\'e}n, 2009a]{BB09a}
Borcea, J. and Br{\"a}nd{\'e}n, P. (2009a).
\newblock The {L}ee-{Y}ang and {P}\'olya-{S}chur programs. {I}. {L}inear
  operators preserving stability.
\newblock {\em Invent. Math.}, 177(3):541--569.

\bibitem[Borcea and Br{\"a}nd{\'e}n, 2009b]{BB09b}
Borcea, J. and Br{\"a}nd{\'e}n, P. (2009b).
\newblock The {L}ee-{Y}ang and {P}\'olya-{S}chur programs. {II}. {T}heory of
  stable polynomials and applications.
\newblock {\em Comm. Pure Appl. Math.}, 62(12):1595--1631.

\bibitem[Cole and Wermer, 1999]{CW99}
Cole, B.~J. and Wermer, J. (1999).
\newblock And{\^{o}}'s theorem and sums of squares.
\newblock {\em Indiana Univ. Math. J.}, 48(3):767--791.

\bibitem[Geronimo and Woerdeman, 2004]{GW04}
Geronimo, J.~S. and Woerdeman, H.~J. (2004).
\newblock Positive extensions, {F}ej\'er-{R}iesz factorization and
  autoregressive filters in two variables.
\newblock {\em Ann. of Math. (2)}, 160(3):839--906.

\bibitem[Knese, 2008]{gK08a}
Knese, G. (2008).
\newblock Bernstein-{S}zeg{\H o} measures on the two dimensional torus.
\newblock {\em Indiana Univ. Math. J.}, 57(3):1353--1376.

\bibitem[Knese, 2010]{gK10}
Knese, G. (2010).
\newblock Rational inner functions in the {S}chur-{A}gler class of the
  polydisk.
\newblock preprint.

\bibitem[Kummert, 1989a]{aK89b}
Kummert, A. (1989a).
\newblock Synthesis of {$3$}-{D} lossless first-order one ports with lumped
  elements.
\newblock {\em IEEE Trans. Circuits and Systems}, 36(11):1445--1449.

\bibitem[Kummert, 1989b]{aK89}
Kummert, A. (1989b).
\newblock Synthesis of two-dimensional lossless {$m$}-ports with prescribed
  scattering matrix.
\newblock {\em Circuits Systems Signal Process.}, 8(1):97--119.

\bibitem[Ruelle, 2010]{dR10}
Ruelle, D. (2010).
\newblock Characterization of {L}ee-{Y}ang polynomials.
\newblock {\em Ann. of Math. (2)}, 171(1):589--603.

\bibitem[Simon, 2005]{bS05}
Simon, B. (2005).
\newblock {\em Orthogonal polynomials on the unit circle. {P}art 1}, volume~54
  of {\em American Mathematical Society Colloquium Publications}.
\newblock American Mathematical Society, Providence, RI.
\newblock Classical theory.

\bibitem[Wagner, 2010]{dW10}
Wagner, D.~G. (2010).
\newblock Multivariate stable polynomials: theory and applications.
\newblock To appear in BAMS.

\end{thebibliography}

\end{document}